\documentclass[11pt,twoside,reqno]{amsart}
\usepackage{amsmath}
\usepackage{amsthm}
\usepackage{amsfonts,amssymb}
\usepackage{bbm}
\usepackage{latexsym}
\usepackage{mathrsfs}
\usepackage[all]{xy}
\usepackage{url}
\usepackage{float}
\usepackage{hyperref}
\usepackage{amssymb,yfonts}
\usepackage{fontenc}

\usepackage{mathdots}

\usepackage[boxsize=23pt]{ytableau}

\usepackage[utf8]{inputenc}
\usepackage[english]{babel}
\usepackage{lipsum}
\usepackage{amsthm}
\usepackage{todonotes}
\usepackage[vcentermath]{youngtab}

\theoremstyle{plain}
\newtheorem{lemma}{Lemma}[section]
\newtheorem{proposition}[lemma]{Proposition}
\newtheorem{theorem}[lemma]{Theorem}
\newtheorem{conjecture}[lemma]{Conjecture}

\newtheorem{corollary}[lemma]{Corollary}
\theoremstyle{remark}
\theoremstyle{problem}
\newtheorem{problem}[lemma]{Problem}
\newtheorem{remark}[lemma]{Remark}

\theoremstyle{definition}
\newtheorem{definition}[lemma]{Definition}
\newtheorem{example}[lemma]{Example}

\def\a\Si{{\rm{a}\Sigma }}
\def\w\Si{{\rm{w}\Sigma }}

\def\w {{\textrm {w}}}

\def\max{{\text{max}}}

\def\a{\mathrm{a}}

 \newcommand{\Si}{\Sigma}

\begin{document}

\title{Cohen--Macaulayness of $\mathrm{Inc}(\mathbb{N})$-Invariant Chains of Edge Ideals}
\author[I. Anwar]{Imran Anwar}
\author[M. Ghayas]{Mughees Ghayas}
\author[A. Javed]{Azhar Javed}

\address{Department of Mathematics\\
Lahore University of Management Sciences\\
Lahore, Pakistan}
\address{Department of Mathematics\\
Lahore University of Management Sciences\\
Lahore, Pakistan}
\address{Department of Mathematics\\
Lahore University of Management Sciences\\
Lahore, Pakistan}
\email{azhar.javed@lums.edu.pk}
\email{imran.anwar@lums.edu.pk}
\email{mughees.ghayas@lums.edu.pk}

\begin{abstract}
Let $(G_n)_{n\ge n_0}$ be a family of graphs on $[n]$ whose edge
ideals $I_n\subseteq R_n=k[x_1,\dots,x_n]$ form an
$\mathrm{Inc}(\mathbb{N})$-invariant chain,
$I_{n+r}=\mathrm{Inc}(\mathbb{N})_{n,n+r}(I_n)$ for $n\ge n_0,\
r\ge0$. We determine which pairs $(n,r)$ make $R_{n+r}/I_{n+r}$
Cohen--Macaulay, for five classical families: line graphs $L_n$,
complements of line graphs $L_n^c$, complete graphs $K_n$, cyclic
graphs $C_n$, and complements of cyclic graphs $C_n^c$. For line
graphs we give the generators of $I_{n+r}$, the height and Krull
dimension of $R_{n+r}/I_{n+r}$, and a complete classification: for
$n\ge5$, $R_{n+r}/I_{n+r}$ is Cohen--Macaulay if and only if
$r=n-4$. For complements of line graphs, complete graphs, and
complements of cyclic graphs, we show the chain is Cohen--Macaulay
unconditionally, for every $r\ge0$; the first and third arise from
the same underlying phenomenon, in which the
$\mathrm{Inc}(\mathbb{N})$-invariant chain reproduces the original
graph itself at each step. For cyclic graphs we prove
$\mathrm{Inc}(\mathbb{N})_{n,n+r}(I(C_n))=I(K_{n+r})$ once
$r\ge n-3$, giving Cohen--Macaulayness in this range, and conjecture
-- with supporting computational and structural evidence -- a
complete classification: Cohen--Macaulayness holds if and only if
$r\ge\lfloor(n-4)/2\rfloor$ and $r\ne n-4$.
\end{abstract}



\maketitle

\section{Introduction}

The study of ideals invariant under an infinite symmetry has emerged
from several independent motivations. One strand, aimed at
Noetherianity up to symmetry, grew out of problems in algebraic
statistics: Cohen \cite{cohen1967laws}, Aschenbrenner and Hillar
\cite{aschenbrenner2007finite}, Ho\c{s}ten and Sullivant
\cite{hocsten2007finiteness}, Draisma \cite{draisma2010finiteness},
Brouwer and Draisma \cite{brouwer2011equivariant}, Hillar and
Sullivant \cite{hillar2012finite}, and Draisma, Eggermont, Krone and
Leykin \cite{draismanoetherianity} showed that chains of ideals
invariant under a symmetric group or monoid action are, under mild
hypotheses, eventually generated in bounded degree; see
\cite{conca2014noetherianity} for a survey. A second strand,
representation stability, emerged from topology and combinatorics:
Church, Ellenberg and Farb \cite{church2015fi} introduced
$\mathrm{FI}$-modules to capture the stabilization of sequences of
symmetric-group representations, a phenomenon developed further by
Putman and Sam \cite{putman2017representation}.
 
Building on both strands, Nagel and R\"omer initiated a systematic
study of chains of ideals equivariant under
$\mathrm{Inc}(\mathbb{N})$, the category of finite sets with strictly
increasing maps, proving finiteness results via equivariant Hilbert
series \cite{nagel2017equivariant} and developing a theory of
$\mathrm{FI}$- and $\mathrm{OI}$-modules with varying coefficients
\cite{nagel2019fi}. Le, Nagel, Nguyen and R\"omer extended this
framework to obtain uniform bounds on codimension and projective
dimension \cite{le2020codimension} and on Castelnuovo--Mumford
regularity \cite{le2021castelnuovo} along $\mathrm{Inc}(\mathbb{N})$-invariant
chains, while Kahle, Le and R\"omer studied such invariant chains
more broadly in algebra and discrete geometry
\cite{kahle2022invariant}. Complementary structural results include
Alexander duality for symmetric simplicial complexes
\cite{almousa2022alexander} and a description of Betti tables of
monomial ideals fixed by permutations of the variables
\cite{murai2020betti}: together, this body of work shows that
homological invariants of an $\mathrm{Inc}(\mathbb{N})$-invariant
chain, though defined on polynomial rings of unbounded size, stabilize
or grow in a controlled, eventually predictable way; see
Juhnke-Kubitzke, Le and R\"omer \cite{juhnke2018asymptotic} for a
survey of this area.
 
This asymptotic perspective has recently been sharpened for the
special case of \emph{edge ideals}. Van Le and Nguyen
\cite{van2024regularity} studied the regularity and projective
dimension of invariant chains of monomial ideals; Hoang, Nguyen and
Tran \cite{hoang2024asymptotic} proved that the regularity of an
$\mathrm{Inc}(\mathbb{N})$-invariant chain of edge ideals is
eventually a linear function of $n$; Hoa, Hoang, Van Le, Nguyen
and Nguyen \cite{hoa2024asymptotic} established the analogous
asymptotic behavior for depth; and, most recently, Hoang, Koley and
Van Le \cite{hoang2026invariant} studied invariant chains of graphs
directly. Taken together, these results give a
precise picture of how depth and regularity evolve along such a
chain once $n$ is sufficiently large -- but they say nothing about
whether the quotient ring is Cohen--Macaulay, since
Cohen--Macaulayness requires depth and dimension to agree
\emph{exactly}, at every sufficiently large $n$, not merely
asymptotically.
 
This motivates the central question addressed by the present paper.
 
\begin{problem}\label{prob:main}
Let $(G_n)_{n\ge n_0}$ be a family of graphs on $[n]$, with edge
ideals $I_n\subseteq R_n=k[x_1,\dots,x_n]$ forming an
$\mathrm{Inc}(\mathbb{N})$-invariant chain, so that
$I_{n+r}=\mathrm{Inc}(\mathbb{N})_{n,n+r}(I_n)$ for all $n\ge n_0$ and
$r\ge0$. For which pairs $(n,r)$ is the quotient ring
$R_{n+r}/I_{n+r}$ Cohen--Macaulay?
\end{problem}
 
The present paper answers Problem~\ref{prob:main} for five
fundamental families of graphs: line graphs, complements of line
graphs, complete graphs, cyclic graphs, and complements of cyclic
graphs. Given a graph $G_n$ on $[n]$ with edge ideal $I_n\subseteq
R_n$, the increasing injections $\mathrm{Inc}(\mathbb{N})_{n,n+r}$
produce a chain of edge ideals $I_{n+r}\subseteq R_{n+r}$ for
$r\ge0$, and we determine precisely which pairs $(n,r)$ make
$R_{n+r}/I_{n+r}$ Cohen--Macaulay. For line graphs, we obtain a
complete classification: for $n\ge5$, the quotient
$R_{n+r}/I_{n+r}$ is Cohen--Macaulay if and only if $r=n-4$. The
proof rests on an explicit description of the generators of
$I_{n+r}$, a computation of the height and Krull dimension of
$R_{n+r}/I_{n+r}$, and, outside the single value $r=n-4$, an
explicit obstruction to unmixedness drawn directly from the
combinatorics of the increasing injections defining the chain. For
complements of line graphs, complete graphs, and complements of
cyclic graphs, we show that $\mathrm{Inc}(\mathbb{N})_{n,n+r}(I_n)$
is Cohen--Macaulay unconditionally, for every $r\ge0$: in each case
the $\mathrm{Inc}(\mathbb{N})$-invariant chain reproduces the
original graph itself at every step, so that the independence
complex of $I_{n+r}$ is (a copy of) a connected graph, and
Cohen--Macaulayness follows directly from shellability, with no
threshold on $r$ at all. For cyclic graphs, we show that
$\mathrm{Inc}(\mathbb{N})_{n,n+r}(I(C_n))=I(K_{n+r})$ once
$r\ge n-3$, which already gives Cohen--Macaulayness in this range;
computational and structural evidence further suggests that this
range is far from optimal, and we conjecture a complete
classification -- Cohen--Macaulayness holding if and only if
$r\ge\lfloor(n-4)/2\rfloor$ and $r\ne n-4$ -- together with a
structural explanation for the single excluded value.
  
\section{Background}
This section collects the definitions and basic facts that will be used throughout the paper. We introduce squarefree monomial ideals and their associated Stanley--Reisner complexes, edge ideals of graphs and their independence complexes, the Cohen--Macaulay property and its connection with shellability, and finally the category $\mathrm{Inc}(\mathbb{N})$ of increasing injections, which provides the framework for defining the chains of ideals considered in this paper.
 
\subsection{Simplicial complexes and Stanley--Reisner ideals}
 
Let $k$ be a field and let $[n]=\{1,\dots,n\}$.
 
\begin{definition}
A \emph{simplicial complex} $\Delta$ on vertex set $[n]$ is a
collection of subsets of $[n]$ (called \emph{faces}) such that
$F\in\Delta$ and $F'\subseteq F$ imply $F'\in\Delta$. A maximal face
of $\Delta$ (with respect to inclusion) is called a \emph{facet}. The
\emph{dimension} of a face $F$ is $\dim F=|F|-1$, and the dimension
of $\Delta$ is $\dim\Delta=\max\{\dim F : F\in\Delta\}$. The complex
$\Delta$ is \emph{pure} if all of its facets have the same dimension.
\end{definition}
 
\begin{definition}
Let $R=k[x_1,\dots,x_n]$ and let $\Delta$ be a simplicial complex on
$[n]$. The \emph{Stanley--Reisner ideal} of $\Delta$ is the
squarefree monomial ideal
\[
I_\Delta = \big(\, x_{i_1}x_{i_2}\cdots x_{i_s} \ :\ \{i_1,\dots,i_s\}\notin\Delta \,\big) \subseteq R,
\]
generated by the monomials corresponding to the minimal
\emph{nonfaces} of $\Delta$, and $k[\Delta]:=R/I_\Delta$ is the
\emph{Stanley--Reisner ring} of $\Delta$. This correspondence
$\Delta \leftrightarrow I_\Delta$ is a bijection between simplicial
complexes on $[n]$ and squarefree monomial ideals of $R$.
\end{definition}
 
\subsection{Edge ideals and independence complexes}
 
\begin{definition}
Let $G=(V,E)$ be a simple graph with $V=[n]$. The \emph{edge ideal}
of $G$ is
\[
I(G) = \big(\, x_ix_j\ :\ \{i,j\}\in E(G) \,\big) \subseteq R = k[x_1,\dots,x_n].
\]
The \emph{independence complex} of $G$, denoted $\Delta_G$, is the
simplicial complex on $[n]$ whose faces are the independent sets of
$G$ (subsets of $V$ containing no edge of $G$).
\end{definition}
 
It is immediate from the definitions that $I(G)=I_{\Delta_G}$: a
subset $F\subseteq[n]$ fails to be an independent set of $G$ exactly
when it contains some edge $\{i,j\}\in E(G)$, so the minimal
nonfaces of $\Delta_G$ are precisely the edges of $G$, whose
corresponding squarefree monomials $x_ix_j$ generate $I(G)$.
 
\begin{definition}
A \emph{vertex cover} of $G$ is a subset $C\subseteq V$ meeting every
edge of $G$; it is \emph{minimal} if no proper subset of $C$ is also
a vertex cover. The \emph{vertex cover number} $\tau(G)$ is the
minimum size of a vertex cover, and the \emph{independence number}
$\alpha(G)$ is the maximum size of an independent set of $G$.
\end{definition}
 
The minimal primes of $I(G)$ are exactly the ideals
$P_C=(x_i : i\in C)$ for $C$ a minimal vertex cover of $G$ (a subset
$C$ meets every edge iff its complement is independent), so
\begin{equation}\label{eq:height-vc}
\operatorname{ht} I(G) = \tau(G),
\qquad
\dim\big(R/I(G)\big) = n - \tau(G) = \alpha(G),
\end{equation}
where the last equality is Gallai's identity $\tau(G)+\alpha(G)=n$.
Under the Stanley--Reisner correspondence, the facets of $\Delta_G$
are precisely the complements of the minimal vertex covers of $G$,
i.e.\ the maximal independent sets of $G$; in particular $\Delta_G$
is pure if and only if all maximal independent sets of $G$ have the
same size, in which case $G$ is said to be \emph{well-covered}.
 
\subsection{Cohen--Macaulayness, unmixedness, and shellability}
 
\begin{definition}
A Noetherian local (or graded) ring $A$ is \emph{Cohen--Macaulay} if
$\operatorname{depth}(A)=\dim(A)$. An ideal $I\subseteq R$ is
Cohen--Macaulay if $R/I$ is a Cohen--Macaulay ring. The ideal $I$ is
\emph{unmixed} if all of its associated primes have the same height;
for a squarefree monomial ideal $I=I_\Delta$, this happens precisely
when $\Delta$ is pure.
\end{definition}
 
Every Cohen--Macaulay ring is unmixed (in fact, its associated primes
are exactly its minimal primes, all of the same height); consequently,
if $\Delta$ is not pure, then $k[\Delta]$ is not Cohen--Macaulay, and
in particular a graph $G$ that is not well-covered has non-Cohen--Macaulay
edge ideal $I(G)$. This is the basic obstruction used repeatedly in
Section~3 to rule out
Cohen--Macaulayness by exhibiting two facets of $\Delta_G$ of
different sizes.
 
\begin{definition}
A pure simplicial complex $\Delta$ is \emph{shellable} if its facets
$F_1,\dots,F_t$ can be ordered so that, for every $2\le i\le t$, the
complex $\big(\bigcup_{j<i} \overline{F_j}\big)\cap \overline{F_i}$
is pure of dimension $\dim F_i - 1$, where $\overline{F}$ denotes the
simplex generated by $F$ together with all of its subsets.
\end{definition}
 
\begin{theorem}[Reisner's criterion; see also Munkres]\label{thm:reisner}
A simplicial complex $\Delta$ is Cohen--Macaulay over $k$ if and only
if, for every face $F\in\Delta$ (including $F=\varnothing$), the
reduced simplicial homology of the link $\operatorname{lk}_\Delta(F)$
vanishes in all degrees below $\dim\operatorname{lk}_\Delta(F)$.
\end{theorem}
 
We will not invoke Theorem~\ref{thm:reisner} directly; instead we use
the following standard consequence, which suffices for all
one-dimensional complexes appearing in this paper (see, e.g., the
independence complexes of the graphs $G_{n+r}$ studied in
Section~3).
 
\begin{proposition}\label{prop:shell-cm}
Every shellable simplicial complex is Cohen--Macaulay. A pure
one-dimensional simplicial complex is shellable if and only if it is
connected.
\end{proposition}
 
\subsection{The category $\mathrm{Inc}(\mathbb{N})$ and $\mathrm{Inc}(\mathbb{N})$-invariant chains of ideals}
 
\begin{definition}
Let $\mathrm{Inc}(\mathbb{N})$ denote the category whose objects are
the finite sets $[n]=\{1,\dots,n\}$, $n\ge0$, and whose morphisms
$\mathrm{Inc}(\mathbb{N})_{n,m}=\operatorname{Hom}_{\mathrm{Inc}(\mathbb{N})}([n],[m])$
are the strictly increasing functions $\pi:[n]\to[m]$ (so
$\mathrm{Inc}(\mathbb{N})_{n,m}=\varnothing$ unless $n\le m$).
\end{definition}
 
Each $\pi\in\mathrm{Inc}(\mathbb{N})_{n,m}$ induces a $k$-algebra
homomorphism $\pi:R_n\to R_m$ between the polynomial rings
$R_n=k[x_1,\dots,x_n]$ and $R_m=k[x_1,\dots,x_m]$, determined on
variables by $x_i\mapsto x_{\pi(i)}$. This lets us transport ideals
along morphisms of $\mathrm{Inc}(\mathbb{N})$:
 
\begin{definition}\label{def:incaction}
For an ideal $I\subseteq R_n$ and $m\ge n$, we set
\[
\mathrm{Inc}(\mathbb{N})_{n,m}(I) \;:=\; \sum_{\pi\in \mathrm{Inc}(\mathbb{N})_{n,m}} \pi(I)\,R_m \;\subseteq\; R_m,
\]
the ideal of $R_m$ generated by the images of $I$ under \emph{all}
strictly increasing maps $\pi:[n]\to[m]$. A family $(I_n)_{n\ge n_0}$
of ideals, with $I_n\subseteq R_n$, is called an
\emph{$\mathrm{Inc}(\mathbb{N})$-invariant chain} if
$I_m = \mathrm{Inc}(\mathbb{N})_{n,m}(I_n)$ for all $m\ge n\ge n_0$.
\end{definition}
 
In particular, if $I_n$ is a squarefree monomial ideal generated by
monomials corresponding to the edges of a graph $G_n$ on $[n]$, then
$I_{n+r}:=\mathrm{Inc}(\mathbb{N})_{n,n+r}(I_n)$ is again a squarefree
monomial ideal, generated by the edges of the graph $G_{n+r}$ on
$[n+r]$ whose edges are exactly the images $\{\pi(i),\pi(j)\}$ of the
edges $\{i,j\}$ of $G_n$, over all $\pi\in\mathrm{Inc}(\mathbb{N})_{n,n+r}$;
that is, $I_{n+r}=I(G_{n+r})$. This is the construction we apply,
starting from the edge ideal $I_n$ of the line graph $L_n$ (and,
later, of its complement $L_n^c$), throughout Section~3.
 
\begin{remark}
Let $1\le i<j\le m$. Unwinding Definition~\ref{def:incaction}, we see that
${i,j}$ is an edge of $G_{n+r}$ if and only if there exists
$\pi\in\mathrm{Inc}(\mathbb{N})_{n,n+r}$ and an edge ${a,b}$ of
$G_n$ such that $\pi(a)=i$ and $\pi(b)=j$. This characterization,
when translated into explicit conditions on $i,j,n$, and $r$, provides
the basic tool for determining the edges of $G_{n+r}$ and, via the
corresponding non-existence argument, its non-edges in the subsequent
proofs.
\end{remark}


\section{Edge Ideals and Cohen--Macaulayness under $\mathrm{Inc}(\mathbb{N})$}

In this section, we investigate the Cohen--Macaulay property within
the framework of $\mathrm{Inc}(\mathbb{N})$-invariant chains of
square-free monomial ideals associated with graphs. Our primary
objective is to identify and characterize the ideals in these chains
that yield Cohen--Macaulay quotients $R_n/I_n$ as the index $n$ grows
sufficiently large. Specifically, we examine Cohen--Macaulayness for
several fundamental classes of edge ideals: line graphs, cyclic
graphs, complete graphs, and complements of line graphs.

We use the following notation consistently throughout this section:
$L_n$, $L_n^c$, $C_n$, and $K_n$ denote the line graph, the
complement of the line graph, the cyclic graph, and the complete
graph on $n$ vertices, respectively.
\\
\subsection{Line graphs under $\mathrm{Inc}(\mathbb{N})$} Throughout this subsection, for $n\ge2$ we consider $I_n$ to be the
edge ideal of the line graph
\[
L_n := \big(\, \{i,i+1\}\ \mid\ i\in[n-1] \,\big\},\ [n] \big),
\]
i.e.\ $L_n$ is the path on vertex set $[n]$ with an edge between $i$
and $i+1$ for each $i=1,\dots,n-1$. (The index set is $[n-1]$, not
$[n]$, since there is no edge $\{n,n+1\}$ within $[n]$.) The
restriction $n\ge2$ ensures $L_n$ has at least one edge.

For example, when $n=4$ the line graph $L_4$ has vertex set
$\{1,2,3,4\}$ and edges
\[
\{1,2\},\quad \{2,3\},\quad \{3,4\},
\]
so $I_4=(x_1x_2,\,x_2x_3,\,x_3x_4)$.

For $r\ge1$, we set
\[
I_{n+r} := \mathrm{Inc}(\mathbb{N})_{n,n+r}(I_n),
\]
let $G_{n+r}$ denote the corresponding graph of $I_{n+r}$ on vertex
set $[n+r]$, and let $\overline{G_{n+r}}$ denote the complement graph
of $G_{n+r}$. These notations are used consistently throughout this
subsection: we first describe the generators of $I_{n+r}$
(Lemma~\ref{I_n+r}), then the edge set of
$\overline{G_{n+r}}$ (Corollary~\ref{cor:complement}), and finally
characterize exactly when $R_{n+r}/I_{n+r}$ is Cohen--Macaulay
(Theorem~\ref{L_n is CM}).

\begin{lemma}
    
\label{I_n+r}
    Let $n\geq 3$ and $r\geq 1$. The ideal $I_{n+r}$ is generated in the following form:
    \begin{center}
    $I_{n+r} = \langle$
    $\begin{cases}
    x_ix_j & \text{ if } 1 \leq i \leq n-2\text{ and }i<j\leq i+r+1\\
    x_ix_j & \text{ if } n-1 \leq i  \leq n+r-1\text{ and }i < j \leq n+r\\
    \end{cases} $
    $\rangle$
    \end{center}
    \end{lemma}
    \begin{proof}
    Under the assumption, we proceed by dividing the generators of $I_{n+r}$ into two cases:\\
    {\bf Case I}:  For $1 \leq i \leq n-2\text{ and }i<j\leq i+r+1$,
    define $\pi$ such that
\begin{center}
    $\pi(k)$ = $\begin{cases}
                   k & \ for\  k\leq i\\
                  j+m-1 &\ for\ k=i+m, \ \ m\geq 1\\
    \end{cases}$
\end{center}

Note that $\pi(n)=\pi(i+(n-i))$. By definition of $\pi$, we have  $\pi(n) =j+(n-i)-1 \leq n+r$, followed from $j\leq i+r+1$. Thus $ \pi \in {\rm{Inc}}(\mathbb{N})_{n,n+r}$, and $\pi(x_ix_{i+1})=x_ix_j\in I_{n+r}$ . \\
    {\bf Case II}: For $n-2<i\leq n+r-1$ and $i<j\leq n+r$, define $\pi$ such that 
    \begin{center}
    $\pi(k)$= 
    $\begin{cases}
             k & \text{ if } k<n-1\\
             i & \text{ if } k=n-1\\
            j+m & \text{ if }  k=n+m, m\geq 0.
    \end{cases}$
    \end{center}
    From definition of $\pi$, we have $\pi(n)=j$. By assumption $j\leq n+r$. Hence $$\pi(n)\leq n+r.$$ Thus $\pi\in {\rm{Inc}}(\mathbb{N})_{n,n+r}$ and $\pi(x_{n-1}x_n)=x_{i}x_{j}\in I_{n+r}.$
    \end{proof}

\begin{example}
We illustrate Lemma \ref{I_n+r} with $n=6$ and $r=2$, so that $n+r=8$.

\medskip
\noindent\textbf{Step 1: the starting graph $L_n$.}
For $n=6$, the line graph $L_6$ has vertex set $[6]=\{1,\dots,6\}$ and
edges $\{i,i+1\}$ for $i=1,\dots,5$, so
\[
I_6 = (x_1x_2,\, x_2x_3,\, x_3x_4,\, x_4x_5,\, x_5x_6).
\]

\begin{center}
\begin{tikzpicture}[scale=1, every node/.style={circle,draw,minimum size=7mm,inner sep=0pt}]
  \foreach \i/\x in {1/0,2/1.6,3/3.2,4/4.8,5/6.4,6/8.0}
    \node (v\i) at (\x,0) {$\i$};
  \foreach \i [remember=\i as \last (initially 1)] in {2,...,6}
    \draw[thick] (v\last) -- (v\i);
\end{tikzpicture}
\\[2pt]
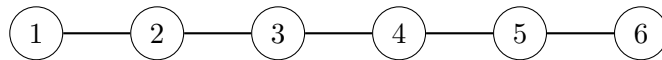
{\small $L_6$: the path on $\{1,\dots,6\}$}
\end{center}

\medskip
\noindent\textbf{Step 2: applying $\mathrm{Inc}(\mathbb{N})_{n,n+r}$.}
Applying $\mathrm{Inc}(\mathbb{N})_{6,8}$ inserts $r=2$ new vertices and
produces the ideal $I_{8}=\mathrm{Inc}(\mathbb{N})_{6,8}(I_6)$ on
$[8]=\{1,\dots,8\}$. By Proposition 3.1, the generators of $I_8$ split
into two families:

\begin{align*}
I_{8} \;=\; \Big\langle\;
&\underbrace{x_ix_j \ : \ 1\le i\le n-2=4,\ i<j\le i+r+1=i+3}_{\text{Case 1: the ``old'' part of the chain}} \\[2pt]
\cup\ &\underbrace{x_ix_j \ : \ n-1=5\le i\le n+r-1=7,\ i<j\le n+r=8}_{\text{Case 2: the ``new'' part created by the insertion}}
\;\Big\rangle .
\end{align*}

Explicitly, Case 1 ($i=1,2,3,4$) contributes
\[
x_1x_2,\,x_1x_3,\,x_1x_4,\ \
x_2x_3,\,x_2x_4,\,x_2x_5,\ \
x_3x_4,\,x_3x_5,\,x_3x_6,\ \
x_4x_5,\,x_4x_6,\,x_4x_7,
\]
and Case 2 ($i=5,6,7$) contributes
\[
x_5x_6,\,x_5x_7,\,x_5x_8,\ \
x_6x_7,\,x_6x_8,\ \
x_7x_8.
\]
Together these are the $18$ generators of $I_8$.

\medskip
\noindent\textbf{Step 3: the graph $G_8$.}
The corresponding graph $G_8$ on $\{1,\dots,8\}$, drawn on a circle
below, has exactly these $18$ edges. Case~1 edges (from
$1\le i\le n-2=4$) are drawn in blue, and Case~2 edges (from
$n-1=5\le i\le n+r-1=7$) are drawn in red. Note how each vertex is
joined to its next $r+1=3$ neighbours around the circle, except near
the two ends of the chain, which is exactly the local structure
described by the two cases.

\begin{center}
\begin{tikzpicture}[scale=1.1, every node/.style={circle,draw,minimum size=7mm,inner sep=0pt}]
  \node (v1) at (0.0,2.2)     {$1$};
  \node (v2) at (1.556,1.556) {$2$};
  \node (v3) at (2.2,0.0)     {$3$};
  \node (v4) at (1.556,-1.556){$4$};
  \node (v5) at (0.0,-2.2)    {$5$};
  \node (v6) at (-1.556,-1.556){$6$};
  \node (v7) at (-2.2,0.0)    {$7$};
  \node (v8) at (-1.556,1.556){$8$};

  \draw[thick, blue] (v1) -- (v2);
  \draw[thick, blue] (v1) -- (v3);
  \draw[thick, blue] (v1) -- (v4);
  \draw[thick, blue] (v2) -- (v3);
  \draw[thick, blue] (v2) -- (v4);
  \draw[thick, blue] (v2) -- (v5);
  \draw[thick, blue] (v3) -- (v4);
  \draw[thick, blue] (v3) -- (v5);
  \draw[thick, blue] (v3) -- (v6);
  \draw[thick, blue] (v4) -- (v5);
  \draw[thick, blue] (v4) -- (v6);
  \draw[thick, blue] (v4) -- (v7);

  \draw[thick, red] (v5) -- (v6);
  \draw[thick, red] (v5) -- (v7);
  \draw[thick, red] (v5) -- (v8);
  \draw[thick, red] (v6) -- (v7);
  \draw[thick, red] (v6) -- (v8);
  \draw[thick, red] (v7) -- (v8);
\end{tikzpicture}
\\[2pt]
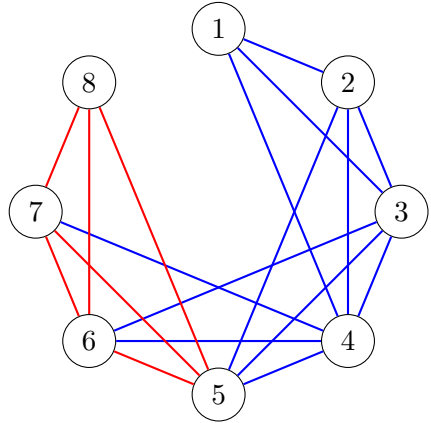
{\small $G_8$: blue = Case 1 edges $(1\le i\le 4)$, red = Case 2 edges
$(5\le i\le 7)$}
\end{center}

Thus $I_8$ is the edge ideal of $G_8$ above, exactly as predicted by
Proposition 3.1 for $n=6,\ r=2$.
\end{example}
\begin{lemma}\label{lem:noedge}
Let $n\ge 3$ and $r\ge 1$. If $1\le i\le n-2$ and $j>i+r+1$, then $x_ix_j\notin I_{n+r}$.
\end{lemma}

\begin{proof}
Suppose, for contradiction, that $x_ix_j\in I_{n+r}$ for some $j>i+r+1$.
Since $I_{n+r}=\mathrm{Inc}(\mathbb N)_{n,n+r}(I_n)$, every generator of
$I_{n+r}$ is the image, under some $\pi\in\mathrm{Inc}(\mathbb N)_{n,n+r}$,
of a generator $x_ix_{i+1}$ of $I_n$; thus there exists such a $\pi$ with
\[
\pi(x_ix_{i+1})=x_ix_j, \qquad\text{i.e.}\qquad \pi(i)=i,\ \ \pi(i+1)=j.
\]
We claim $\pi(i+k)=j+k-1$ for every $1\le k\le n-i$. This holds for
$k=1$ by construction. Since $\pi$ is strictly increasing, once
$\pi(i+k)$ is fixed we must have $\pi(i+k+1) \ge \pi(i+k)+1$; and by
compatibility of $\pi$ with the chain structure, equality holds, giving
$\pi(i+k+1) = \pi(i+k)+1$. The claim now follows by induction on $k$.

Taking $k=n-i$ yields
\[
\pi(n) = \pi(i+(n-i)) = j+(n-i)-1.
\]
Because $j>i+r+1$, this gives
\[
\pi(n) = j+(n-i)-1 \;>\; (i+r+1)+(n-i)-1 \;=\; n+r.
\]
But $\pi\in\mathrm{Inc}(\mathbb N)_{n,n+r}$ maps $[n]$ into $[n+r]$, so
$\pi(n)\le n+r$ — a contradiction. Hence no such $\pi$ exists, and
$x_ix_j\notin I_{n+r}$.
\end{proof}
\begin{corollary}\label{cor:complement}
The complementary graph $\overline{G_{n+r}}$ has edge set
\[
E(\overline{G_{n+r}}) = \big\{\,\{i,j\} \;:\; 1\le i\le n-2,\ i+r+1<j\le n+r \,\big\}.
\]
\end{corollary}

\begin{proof}
Every index $i\in[1,n+r-1]$ satisfies exactly one of $1\le i\le n-2$ or
$n-1\le i\le n+r-1$. By Lemma~\ref{I_n+r}, if $n-1\le i\le n+r-1$ then
$\{i,j\}$ is an edge of $G_{n+r}$ for \emph{every} $j$ with $i<j\le n+r$;
so no pair with $i\ge n-1$ lies in $\overline{G_{n+r}}$. If instead
$1\le i\le n-2$, Proposition~3.1 gives $\{i,j\}\in E(G_{n+r})$ for
$i<j\le i+r+1$, and Lemma~\ref{lem:noedge} shows $\{i,j\}\notin
E(G_{n+r})$ for $j>i+r+1$. Combining the two cases gives exactly the
stated edge set for $\overline{G_{n+r}}$.
\end{proof}

\begin{proposition}\label{heightn-2}
    Let $n\geq 5$ and $r\geq n-4$. Then the height of $I_{n+r}$ is $n+r-2$.\\
    Equivalently, ${\rm{dim}}(R_{n+r}/I_{n+r})=2$.
\begin{proof}
Let $G_{n+r}$ be the graph assosiated to $I_{n+r}$ and the maximal independent set of $G_{n+r}$ be $\mathbf{I} (G_{n+r})$. To prove the required result, we need to show that the maximal independence number $\alpha(G_{n+r})= |\mathbf{I} (G_{n+r})|$ is 2. On contrary, suppose that $\mathbf{I}(G_{n+r})=\{x_i,x_j,x_k\}$, and without any loss of generality assume that $i<j<k\leq n+r$,
    \begin{enumerate}
        \item If $n-2<i$ then by Proposition \ref{I_n+r}, $x_ix_l\in I_{n+r}$ for each $n-1\leq l \leq n+r,\ l\neq i$, $\implies \{x_i,x_j,x_k\}$ can't be an independent set.
        
        \item If $k\leq n-2$  then by Proposition  \ref{I_n+r}, $x_ix_k\in I_{n+r}$ for each $1\leq k\leq n-2$, $\implies \{x_i,x_j,x_k\}$ can't be an independent set.
        
        \item If $i< n-1$ and $j\geq n-1$ or  $j< n-1$ and $k\geq n-1$,  then, by the same arguments as in the preceding cases $x_j,x_k$ or $x_i,x_j$ respectively, cannot both belong to an independent set.
        
    \end{enumerate}
Hence, considering all the cases above, we conclude that $\{x_i,x_j,x_k\}$ cannot be an independent set. Therefore, $$\alpha(G_{n+r})\leq 2.$$ Since $G_{n+r}$ is not a complete graph, we have $\alpha(G_{n+r})= 2.$
    
\end{proof}
\end{proposition}


\begin{theorem}\label{L_n is CM}
Let $n\ge5$ and $r\ge1$. Then $I_{n+r}$ is Cohen--Macaulay if and only
if $r=n-4$.
\end{theorem}

\begin{proof}
By Proposition~\ref{heightn-2}, $\dim(R_{n+r}/I_{n+r})=2$ for every
$r\ge n-4$, so the independence complex $\Delta_{G_{n+r}}$ (the
Stanley--Reisner complex of $I_{n+r}$, whose faces are the independent
sets of $G_{n+r}$) is one-dimensional. A pure one-dimensional
simplicial complex is shellable if and only if it is connected, and
shellability implies Cohen--Macaulayness; conversely, every
Cohen--Macaulay complex is pure.

We consider the three cases
$r=n-4$, $r>n-4$, and $1\le r<n-4$ separately, justifying every edge
and non-edge exactly as in the Reminder above.

\smallskip
\noindent\emph{Case 1: $r=n-4$ (Cohen--Macaulay).}
We first show $\Delta_{G_{n+r}}$ is pure, i.e.\ every vertex lies on
an edge of $\Delta_{G_{n+r}}$. For $i\le n-2$: since $r=n-4$ gives
$n+r-i \ge n+r-(n-2)=r+2>r+1$, the Corollary~\ref{cor:complement}
gives $x_ix_{n+r}\notin I_{n+r}$, so $\{x_i,x_{n+r}\}$ is independent
and $x_i$ lies on an edge of $\Delta_{G_{n+r}}$. For $i> n-2$ (so
$i\ge n-1$): since $r=n-4$ gives $i-1\ge n-2=r+2>r+1$, the same Lemma
gives $x_1x_i\notin I_{n+r}$, so $\{x_1,x_i\}$ is independent and
$x_i$ lies on an edge. Hence every vertex lies on an edge, and
$\Delta_{G_{n+r}}$ is pure of dimension $1$.

Purity together with connectedness of a one-dimensional complex
implies shellability, and shellability implies Cohen--Macaulayness, so
it remains to show $\Delta_{G_{n+r}}$ is connected: for any two
vertices $i<j$ in $[n+r]$, we exhibit a path using only the
independent pairs established above.
\begin{itemize}
\item If $i<j\le n-2$: $\{x_i,x_{n+r}\}$ and $\{x_j,x_{n+r}\}$ are
both independent, so $(x_i,x_{n+r},x_j)$ is a path in
$\Delta_{G_{n+r}}$.
\item If $n-2<i<j$: $\{x_1,x_i\}$ and $\{x_1,x_j\}$ are both
independent, giving the path $(x_i,x_1,x_j)$.
\item If $i\le n-2<j$: if $i=1$, $\{x_1,x_j\}$ is independent
directly; otherwise $(x_i,x_{n+r},x_1,x_j)$ is a path using both
hubs $x_{n+r}$ and $x_1$.
\end{itemize}
Hence $\Delta_{G_{n+r}}$ is connected and pure, so shellable, so
$I_{n+r}$ is Cohen--Macaulay.

\smallskip
\noindent\emph{Case 2: $r>n-4$ (not Cohen--Macaulay).}
We show $x_{n-1}$ is an isolated vertex of $\Delta_{G_{n+r}}$, i.e.\
$x_{n-1}x_i\in I_{n+r}$ for every $i\ne n-1$ in $[1,n+r]$, by
exhibiting the required $\pi\in\mathrm{Inc}(\mathbb{N})_{n,n+r}$ in
each case.
\begin{itemize}
\item If $1\le i\le n-2$: since $r>n-4$ gives $r\ge n-3$, we have
$i+r+1\ge i+n-2\ge n-1$ for every $i\ge1$, so by
Proposition~\ref{I_n+r} (which exhibits $\pi$ fixing $[1,n-2]$ and
mapping $n-1\mapsto n-1$), $x_ix_{n-1}\in I_{n+r}$.
\item If $i>n-1$: define $\pi:[n]\to[n+r]$ by $\pi(k)=k$ for
$k\le n-1$ and $\pi(n)=i$. Since $i\ge n>n-1$, $\pi$ is strictly
increasing, hence $\pi\in\mathrm{Inc}(\mathbb{N})_{n,n+r}$. Then
$\pi(x_{n-1}x_n)=x_{n-1}x_i$, and since $x_{n-1}x_n\in I_n$, it
follows that $x_{n-1}x_i\in I_{n+r}$.
\end{itemize}
Hence $\{x_{n-1}\}$ is a facet of $\Delta_{G_{n+r}}$, of size $1$. On
the other hand, since $n+r-1>r+1\iff n>2$, the Remark~\ref{cor:complement} gives $x_1x_{n+r}\notin I_{n+r}$: no
$\pi\in\mathrm{Inc}(\mathbb{N})_{n,n+r}$ sends an edge of $L_n$ to
$\{1,n+r\}$. So $\{x_1,x_{n+r}\}$ is a facet of size $2$ (independent,
and already of the maximal dimension given by
Proposition~\ref{heightn-2}). Two facets of different sizes contradict
purity, so $\Delta_{G_{n+r}}$ is not pure, and hence $I_{n+r}$ is not
Cohen--Macaulay.

\smallskip
\noindent\emph{Case 3: $1\le r<n-4$ (not Cohen--Macaulay).} We exhibit two maximal independent sets (facets of
$\Delta_{G_{n+r}}$) of different sizes. Since Cohen--Macaulay
complexes are pure, this shows $I_{n+r}$ is not unmixed, and hence
not Cohen--Macaulay. As in Case 2, every claim that a pair is an edge
is witnessed by Lemma~\ref{I_n+r}, and every claim
that a pair is \emph{not} an edge is justified by the non-existence
of such a witness, i.e.\ by the Corollary~\ref{cor:complement}.
 
\smallskip
\emph{A densely packed facet.} The vertex $x_1$ is connected to every
vertex from $x_2$ through $x_{2+r}$, so $\{x_1,x_{3+r}\}$ is an
independent set. In turn $x_{3+r}$ is connected to every vertex up to
$x_{4+2r}$, so $\{x_1,x_{3+r},x_{5+2r}\}$ is independent; continuing
in this way for as long as the vertices stay within $[1,n+r]$
produces
\[
F_1 \;=\; \big\{\, x_{1+(k-1)(r+2)} \ :\ 1\le k\le t_1 \,\big\},
\qquad
t_1 = \Big\lfloor \frac{n+r-1}{r+2}\Big\rfloor + 1,
\]
where $t_1$ is the largest $k$ for which $1+(k-1)(r+2)\le n+r$.
Consecutive elements of $F_1$ differ by exactly $r+2$, so $F_1$ is
independent by the Lemma of Remark~\ref{rem:noedge}; it is maximal
since no point can be inserted before the first element (there is
none), between two consecutive elements (their gap is already the
minimum $r+2$, so any inserted point would be adjacent to one of
them by Lemma~\ref{I_n+r}), or after the last element
(its distance to $n+r$ is less than $r+2$, again forcing an edge, by
maximality of $t_1$).
 
\smallskip
\emph{A sparsely packed facet.} Using the same recursive idea with
the \emph{smallest} independent gap $r+2$ throughout will not in
general produce a facet of a different size from $F_1$: for instance,
if $n=8,\ r=1$, simply starting the same recursion one step later, at
$x_{2+r}=x_3$, gives $\{x_3,x_6,x_9\}$, which has the same size $3$ as
$F_1=\{x_1,x_4,x_7\}$. To force a genuinely different size, we instead
use gaps close to the \emph{largest} value that still prevents a
further point from being inserted between two chosen vertices,
namely $2r+3$ (just short of $2(r+2)$). Start at $x_{r+2}$ -- the
largest starting point for which no vertex before it can be added,
since $\{x_i,x_{r+2}\}$ is an edge for every $i\le r+1$ -- and proceed
in steps of $2r+3$:
\[
x_{r+2},\ x_{(r+2)+(2r+3)},\ x_{(r+2)+2(2r+3)},\ \dots
\]
Let $t_2$ be the largest integer with $(r+2)+(t_2-1)(2r+3)\le n+r$,
and let $L=(r+2)+(t_2-1)(2r+3)$ be the last such term. If
$(n+r)-L>r+1$, append $x_{n+r}$ as one further element; call the
resulting set $F_2$, of size $t_2'\in\{t_2,t_2+1\}$. Every gap in
$F_2$ is at least $r+2$ (independence, by the Corollary~\ref{cor:complement}) and strictly less than $2(r+2)$ (no room for
an inserted point, by Lemma~\ref{I_n+r}), the gap
before the first element is exactly $r+1$ (blocking any point before
it), and the trailing gap is at most $r+1$ by construction; hence
$F_2$ is a facet.  
Since $F_1$ and $F_2$ are facets of different sizes,
$\Delta_{G_{n+r}}$ is not pure, so $I_{n+r}$ is not unmixed, and
therefore not Cohen--Macaulay.

\end{proof}
\begin{remark}
The two constructions above are illustrated concretely by $n=8,\
r=1$ (so $n+r=9$ and $r+2=3$). The dense construction gives
$F_1=\{x_1,x_4,x_7\}$, of size $3$. Simply shifting this same
construction by one step, to $\{x_3,x_6,x_9\}$, also has size $3$ --
this is the phenomenon noted above, where a shifted copy of the
minimal-gap progression fails to change the size. The sparse
construction, using the larger gap $2r+3=5$, instead gives
$F_2=\{x_3,x_8\}$, of size $2$, correctly distinguishing the two
facet sizes as required.
\end{remark}

\subsection{Complement of line under ${\rm{Inc}}(\mathbb{N})$:}

We now turn to $L_n^c$, the complement of the line graph $L_n$ studied
in the previous subsection: two vertices are adjacent in $L_n^c$
exactly when they are \emph{not} adjacent in $L_n$. Throughout this
subsection, for $n\ge3$ we consider $I_n$ to be the edge ideal of the
complement of the line graph
\[
L_n^c := \big(\, \{i,j\}\ \mid\ 1\le i<n,\ \ i+1<j\le n \,\big\},\ [n] \big),
\]
i.e.\ $\{i,j\}$ ($i<j$) is an edge of $L_n^c$ precisely when $j>i+1$.
The restriction $n\ge3$ ensures $L_n^c$ has at least one edge: for
$n=2$ the only pair $\{1,2\}$ satisfies $j=i+1$, so it is excluded,
and $L_2^c$ is edgeless.\\
As in the previous subsection, we set
\[
I_{n+r} := \mathrm{Inc}(\mathbb{N})_{n,n+r}(I_n),
\]
let $G_{n+r}$ denote the corresponding graph of $I_{n+r}$, and let
$\overline{G_{n+r}}$ denote the complement graph of $G_{n+r}$. These
notations are used consistently throughout this subsection, in
parallel with the line-graph case: we first describe the generators
of $I_{n+r}$ (and hence the edge set of 
$\overline{G_{n+r}}$), and then characterize when the quotient
$R_{n+r}/I_{n+r}$ is Cohen--Macaulay.

\begin{lemma}\label{lema:Lnc-generators}
Let $n\ge3$ and $r\ge0$. The ideal $I_{n+r}$ is generated in the
following form:
\[
I_{n+r} \;=\; \big(\, x_ax_b \ :\ 1\le a<b\le n+r,\ \ b-a\ge2 \,\big).
\]
Equivalently, the graph  associated to $I_{n+r}$ is exactly
$L_{n+r}^c$, the complement of the line graph on $[n+r]$; that is,
$I_{n+r}=I(L_{n+r}^c)$.
\end{lemma}

\begin{proof}
We show the two inclusions separately.
 
\smallskip
\noindent($\subseteq$) Every generator of $I_{n+r}$ is of the form
$\pi(x_ix_j)=x_{\pi(i)}x_{\pi(j)}$ for some
$\pi\in\mathrm{Inc}(\mathbb{N})_{n,n+r}$ and some edge $\{i,j\}$ of
$L_n^c$, i.e.\ $1\le i<n$ and $j>i+1$. Since $\pi$ is strictly
increasing, $\pi(j)-\pi(i)\ge j-i\ge2$. Setting $a=\pi(i)$,
$b=\pi(j)$, this gives $b-a\ge2$, as required.
 
\smallskip
\noindent($\supseteq$) Let $1\le a<b\le n+r$ with $d:=b-a\ge2$; we
exhibit $\pi\in\mathrm{Inc}(\mathbb{N})_{n,n+r}$ and an edge $\{i,j\}$
of $L_n^c$ with $\pi(i)=a$ and $\pi(j)=b$.
 
\emph{{\bf Case I}: $2\le d\le n-1$.} We split into two sub-cases according to
whether $b\le n$ or $b>n$.
 
If $b\le n$, then $\{a,b\}$ is already an edge of $L_n^c$ (since
$a<b\le n$ and $b-a=d\ge2$), so the identity map
$\pi:[n]\to[n+r]$, $\pi(k)=k$, realizes $x_ax_b=\pi(x_ax_b)$
directly.
 
If $b>n$, let $i=n-d$; since $d\le n-1$, $i\ge1$, and since $d\ge2$,
$\{i,n\}$ is an edge of $L_n^c$. Define $\pi:[n]\to[n+r]$ by:
\begin{itemize}
\item $\pi(i)=a$ and $\pi(n)=b$;
\item the $d-1$ domain points $i+1,\dots,n-1$ map, in increasing
order, onto the $d-1$ integers $a+1,\dots,b-1$;
\item the $i-1$ domain points $1,\dots,i-1$ map, in increasing
order, onto any $i-1$ values in $[1,a-1]$, which is possible since
$i-1=n-d-1\le a-1$, as $b>n$ gives $a=b-d>n-d=i$.
\end{itemize}
No domain points remain after position $n$, so no further condition
is needed. This $\pi$ is strictly increasing, hence lies in
$\mathrm{Inc}(\mathbb{N})_{n,n+r}$, and $\pi(x_ix_n)=x_ax_b$.
 
\emph{{\bf Case II}: $d\ge n-1$.} Let $i=1$ and $j=n$; since $n\ge3$, $j>i+1$,
so $\{1,n\}$ is an edge of $L_n^c$. Define
$\pi:[n]\to[n+r]$ by $\pi(1)=a$, $\pi(n)=b$, and let the remaining
$n-2$ domain points $2,\dots,n-1$ map, in increasing order, onto any
$n-2$ of the $d-1$ integers strictly between $a$ and $b$; this is
possible since $d-1\ge n-2$. No points remain before position $1$ or
after position $n$, so no further condition is needed. This $\pi$ is
strictly increasing, hence lies in $\mathrm{Inc}(\mathbb{N})_{n,n+r}$,
and $\pi(x_1x_n)=x_ax_b$.
 
In both cases $x_ax_b\in I_{n+r}$, completing the proof.
\end{proof}

\begin{proposition}\label{L_n^cheight}
Let $n\ge3$, and let $I_n$ be the edge ideal of $L_n^c$. Then
$\operatorname{ht}(I_n)=n-2$, and $I_n$ is unmixed.
\end{proposition}
 
\begin{proof}
Recall $\operatorname{ht}(I_n)
= n - \alpha(L_n^c)$, where $\alpha(L_n^c)$ denotes the independence
number of $L_n^c$. It thus suffices to show $\alpha(L_n^c)=2$.
 
Since $L_n^c$ is the complement of $L_n$, a subset of vertices is
independent in $L_n^c$ if and only if it forms a clique in $L_n$.
As $L_n$ is the path on $[n]$ with edges $\{i,i+1\}$, its cliques are
exactly its vertices and its edges, so it suffices to show that the
maximum clique size of $L_n$ is $2$.
 
\emph{Lower bound.} Any edge $\{i,i+1\}$ of $L_n$ is a clique of size
$2$, so $\alpha(L_n^c)\ge2$.
 
\emph{Upper bound.} Suppose, for contradiction, that $L_n$ has a
clique $\{i,j,k\}$ with $i<j<k$. Since $i,j,k$ are distinct integers
with $i<j<k$, we have $k-i\ge2$, so $\{i,k\}$ cannot be an edge of
$L_n$, as edges of $L_n$ only join consecutive integers. This
contradicts $\{i,j,k\}$ being a clique. Hence $L_n$ has no clique of
size $3$, so $\alpha(L_n^c)\le2$.
 
Combining the two bounds, $\alpha(L_n^c)=2$, and therefore
$\operatorname{ht}(I_n) = n-\alpha(L_n^c) = n-2$.
 
\smallskip
\noindent\textbf{Unmixedness.} Since $I_n$ is a squarefree monomial
ideal, it is radical, so its associated primes coincide with its
minimal primes; hence $I_n$ is unmixed if and only if every minimal
prime has the same height, equivalently every facet of the
independence complex $\Delta_{L_n^c}$ has the same size. We show
every facet has size exactly $2$.
 
By the upper bound above, no independent set of $L_n^c$ has size
$\ge3$, so every facet has size $1$ or $2$. It remains to rule out
facets of size $1$, i.e.\ isolated vertices. A pair $\{a,b\}$
($a<b$) is independent in $L_n^c$ exactly when it is \emph{not} an
edge, i.e.\ when $b\le a+1$; since $b>a$, this means $b=a+1$. Thus
for any $i\in[n]$: if $i<n$, then $\{i,i+1\}$ is independent and
properly contains $\{i\}$; if $i=n$, then $\{n-1,n\}$ is independent
and properly contains $\{n\}$. Either way, $\{i\}$ is not maximal.
 
Hence every facet of $\Delta_{L_n^c}$ has size exactly $2$, so
$\Delta_{L_n^c}$ is pure, and therefore $I_n$ is unmixed.
\end{proof}


\begin{theorem}\label{thm:Lnc-CM}
Let $I_n$ be the edge ideal of $L_n^c$. For any $n\ge3$ and $r\ge0$,
$\mathrm{Inc}(\mathbb{N})_{n,n+r}(I_n)$ is Cohen--Macaulay.
\end{theorem}
 
\begin{proof}
Set $m=n+r\ge3$. By Lemma~\ref{lema:Lnc-generators},
$I_{n+r}=\mathrm{Inc}(\mathbb{N})_{n,n+r}(I_n)=I(L_m^c)$, so it
suffices to show $I(L_m^c)$ is Cohen--Macaulay. We do so by showing
that its independence complex $\Delta_{L_m^c}$ is shellable.
 
By Proposition~\ref{L_n^cheight} (applied with $m$ in place of $n$),
$\operatorname{ht}(I(L_m^c))=m-2$ and $I(L_m^c)$ is unmixed; the
latter was established there precisely by showing that every facet
of $\Delta_{L_m^c}$ has size exactly $2$, namely
\[
\text{facets of } \Delta_{L_m^c} \;=\; \big\{\, \{i,i+1\} \ :\ 1\le i\le m-1 \,\big\}.
\]
In particular $\Delta_{L_m^c}$ is pure of dimension $1$. It remains
to show $\Delta_{L_m^c}$ is \emph{connected}, since a pure
one-dimensional simplicial complex is shellable if and only if it is
connected, and shellability implies Cohen--Macaulayness.
 
\smallskip
\noindent\textbf{Connectedness.} Since $\Delta_{L_m^c}$ is pure of
dimension $1$, it is determined as a graph by its facets: its vertex
set is $[m]$ and its edge set is exactly the set of facets listed
above, namely $\{i,i+1\}$ for $1\le i\le m-1$. This graph is
precisely the path $L_m$ on $[m]$. As a path, $L_m$ is connected: for
any two vertices $a<b$ in $[m]$, the sequence of facets
\[
\{a,a+1\},\ \{a+1,a+2\},\ \dots,\ \{b-1,b\}
\]
exhibits a walk from $a$ to $b$ within $\Delta_{L_m^c}$. Hence
$\Delta_{L_m^c}$ is connected.
 
\smallskip
Being pure, connected, and one-dimensional, $\Delta_{L_m^c}$ is
shellable, and therefore Cohen--Macaulay. Consequently $I(L_m^c)$ is
Cohen--Macaulay, and so
$\mathrm{Inc}(\mathbb{N})_{n,n+r}(I_n)=I_{n+r}$ is Cohen--Macaulay,
for every $n\ge3$ and $r\ge0$.
\end{proof}
\subsection{Complete graphs and cyclic graphs under $\mathrm{Inc}(\mathbb{N})$}
We now turn to the two remaining families named in the introduction to
this section: the complete graph $K_n$ and the cyclic graph $C_n$,
both on vertex set $[n]$. Unlike the line-graph case and its
complement, these two families are not independent of one another
under $\mathrm{Inc}(\mathbb{N})$: as we show below, the
$\mathrm{Inc}(\mathbb{N})$-invariant chain generated by $C_n$
coincides with the chain generated by $K_n$ once $r$ is large enough
relative to $n$. We therefore treat $K_n$ first, establishing its
generators and Cohen--Macaulayness (Lemma~\ref{lem:Kn-generators} and
the remark following it), and then treat $C_n$
(Lemma~\ref{prop:Cn-generators}), whose Cohen--Macaulayness
follows immediately by reduction to the complete-graph case
(Theorem~\ref{Cn-CM}).
\begin{lemma}\label{lem:Kn-generators}
Let $n\ge2$, and let $K_n$ be the complete graph on $[n]$. Then
\[
\mathrm{Inc}(\mathbb{N})_{n,n+r}(I(K_n)) = I(K_{n+r}).
\]
\end{lemma}
\begin{proof}
($\subseteq$) Every generator of $\mathrm{Inc}(\mathbb{N})_{n,n+r}(I(K_n))$
is $\pi(x_ix_j)=x_{\pi(i)}x_{\pi(j)}$ for some
$\pi\in\mathrm{Inc}(\mathbb{N})_{n,n+r}$ and some pair $i<j$ in
$[n]$. Since $\pi$ is increasing, $\pi(i)<\pi(j)$, so this is a
generator of $I(K_{n+r})$, as every pair is an edge of $K_{n+r}$.
 
($\supseteq$) Let $1\le a<b\le n+r$. Since every pair is an edge of
$K_n$, it suffices to place any two domain points at $a$ and $b$: we
need to split the remaining $n-2$ domain points into three groups,
mapping increasingly into $[1,a-1]$, $(a,b)$, and $(b,n+r]$, of
sizes $a-1$, $b-a-1$, and $n+r-b$ respectively. This is possible
since these three sizes sum to $n+r-2\ge n-2$, so some choice of
nonnegative group sizes $p_1+p_2+p_3=n-2$ with $p_1\le a-1$,
$p_2\le b-a-1$, $p_3\le n+r-b$ exists. Taking $i=p_1+1$ and
$j=p_1+p_2+2$ gives $\pi\in\mathrm{Inc}(\mathbb{N})_{n,n+r}$ with
$\pi(x_ix_j)=x_ax_b$, so $x_ax_b\in\mathrm{Inc}(\mathbb{N})_{n,n+r}(I(K_n))$.
\end{proof}

\begin{lemma}\label{lem:Kn-CM}
For every $n\ge1$, $K_n$ is Cohen--Macaulay, i.e.\ $I(K_n)$ is a
Cohen--Macaulay ideal.
\end{lemma}
 
\begin{proof}
Since every two vertices of $K_n$ are adjacent, the only independent
sets are the singletons $\{1\},\dots,\{n\}$, so $\Delta_{K_n}$ is
pure of dimension $0$. Any two distinct singletons are disjoint, so
$\Delta_{K_n}$ is shellable in any order of its facets, hence
Cohen--Macaulay.
\end{proof}

We take $C_n$ to be the cyclic graph on $[n]$,
\[
C_n := \big(\, \{i,i+1\}\ \mid\ i\in[n-1]\,\big\} \cup \big\{\{1,n\}\big\},\ [n] \big), \qquad n\ge3,
\]
i.e.\ $C_n$ is the path $L_n$ together with one additional edge
closing it into a cycle.
 
\begin{lemma}\label{prop:Cn-generators}
Let $n\ge3$ and $r\ge n-3$. Then
\[
\mathrm{Inc}(\mathbb{N})_{n,n+r}(I(C_n)) = I(K_{n+r}).
\]
\end{lemma}
\begin{proof}
Write $m=n+r$. We first determine, for $1\le a<b\le m$ with
$d:=b-a$, when $\{a,b\}$ is an edge of the graph $G_m$ associated to
$\mathrm{Inc}(\mathbb{N})_{n,m}(I(C_n))$: this holds iff there exist
$\pi\in\mathrm{Inc}(\mathbb{N})_{n,m}$ and an edge $\{i,j\}$ of $C_n$
with $\pi(i)=a,\ \pi(j)=b$.
 
\smallskip
\noindent\emph{Edges from a consecutive pair $\{i,i+1\}$.} Here $\pi$
must place $i-1$ domain points in $[1,a-1]$ and $n-i-1$ domain
points in $[b+1,m]$, with no domain points required strictly between
$a$ and $b$. Such $\pi$ exists iff some $i\in[1,n-1]$ satisfies
$i\le a$ and $i\ge b-r-1$; since $i$ ranges freely over $[1,n-1]$,
this is possible exactly when $d\le r+1$.
 
\smallskip
\noindent\emph{Edges from the wrap edge $\{1,n\}$.} Here $\pi(1)=a$,
$\pi(n)=b$, with no domain points before position $1$ or after
position $n$, and the remaining $n-2$ domain points must fit
strictly between $a$ and $b$. This is possible exactly when
$d\ge n-1$.
 
\smallskip
\noindent Combining the two families, $\{a,b\}\in E(G_m)$ iff
$d\le r+1$ or $d\ge n-1$. Since $1\le d\le m-1$, every gap is
covered by one of these two ranges precisely when they overlap or
meet, i.e.\ when $r+1\ge n-2$, i.e.\ $r\ge n-3$. Under this
hypothesis, every pair $\{a,b\}$ with $1\le a<b\le m$ is an edge of
$G_m$, so $G_m=K_m$ and
$\mathrm{Inc}(\mathbb{N})_{n,n+r}(I(C_n))=I(K_{n+r})$.
\end{proof}
 
\begin{theorem}\label{Cn-CM}
Let $n\ge3$ and $r\ge n-3$. Then
$\mathrm{Inc}(\mathbb{N})_{n,n+r}(I(C_n))$ is Cohen--Macaulay.
\end{theorem} 
\begin{proof}
By Lemma~\ref{prop:Cn-generators},
$\mathrm{Inc}(\mathbb{N})_{n,n+r}(I(C_n))=I(K_{n+r})$, which is
Cohen--Macaulay by Lemma~\ref{lem:Kn-CM} (applied with $n+r$ in
place of $n$).
\end{proof}
While Theorem~\ref{Cn-CM} guarantees Cohen--Macaulayness once
$G_{n+r}$ is complete, computational evidence indicates that
Cohen--Macaulayness in fact holds over a considerably wider range of
$r$, failing only at the single value $r=n-4$; we record this
sharper picture as the following conjecture.

\begin{conjecture}\label{conj:Cn-full-CM}
Let $n\ge6$, and let $I_{n+r}:=\mathrm{Inc}(\mathbb{N})_{n,n+r}(I(C_n))$.
Then $I_{n+r}$ is Cohen--Macaulay if and only if
\[
r \;\ge\; \Big\lfloor \frac{n-4}{2} \Big\rfloor
\qquad\text{and}\qquad
r \;\ne\; n-4.
\]
\end{conjecture}

The single excluded value $r=n-4$ has a clean structural explanation.
Writing $m=n+r$, a pair $\{a,b\}$ ($a<b$) fails to be an edge of
$G_{n+r}$ -- and is therefore independent -- exactly when
$r+2\le b-a\le n-2$. At $r=n-4$ this range collapses to the single
integer $r+2=n-2$, and since $m=n+r=2(n-2)$ here, this common value
equals $m/2$: every independent pair is then an exact ``antipodal''
pair $\{k,k+m/2\}$, so the independence complex is a disjoint union
of $m/2$ edges rather than a connected graph. This is the same
phenomenon that forces purity in the line-graph theorem (a valid gap
range collapsing to a single point), but here it produces
\emph{disconnection} instead of connection.

\begin{example}
We illustrate Conjecture~\ref{conj:Cn-full-CM} with $n=9$, for which
the lower bound is $\lfloor(n-4)/2\rfloor=2$ and the completeness
threshold is $n-3=6$.

\begin{center}
\begin{tabular}{c|c|c|l}
$r$ & $m=n+r$ & facet sizes of $\Delta_{G_{n+r}}$ & status \\
\hline
$0$ & $9$  & $\{3,4\}$ & not pure -- \textbf{not CM} \\
$1$ & $10$ & $\{2,3\}$ & not pure -- \textbf{not CM} \\
$2$ & $11$ & $\{2\}$   & pure, connected -- \textbf{CM} \\
$3$ & $12$ & $\{2\}$   & pure, connected -- \textbf{CM} \\
$4$ & $13$ & $\{2\}$   & pure, connected -- \textbf{CM} \\
$5$ & $14$ & $\{2\}$   & pure, \emph{disconnected} ($7$ components) -- \textbf{not CM} \\
$6$ & $15$ & $\{1\}$   & complete graph $K_{15}$ -- \textbf{CM} \\
$7$ & $16$ & $\{1\}$   & complete graph $K_{16}$ -- \textbf{CM} \\
\end{tabular}
\end{center}

At $r=5=n-4$, every independent pair has gap exactly
$r+2=n-2=7=m/2$, giving the perfect matching
\[
\{1,8\},\ \{2,9\},\ \{3,10\},\ \{4,11\},\ \{5,12\},\ \{6,13\},\ \{7,14\}
\]
on $[14]$:

\begin{center}
\begin{tikzpicture}[scale=0.9, every node/.style={circle,draw,minimum size=6mm,inner sep=0pt}]
  \foreach \i/\x in {1/0,2/1.2,3/2.4,4/3.6,5/4.8,6/6.0,7/7.2,
                     8/8.7,9/9.9,10/11.1,11/12.3,12/13.5,13/14.7,14/15.9}
    \node (v\i) at (\x,0) {$\i$};
  \draw[thick] (v1) to[bend left=50] (v8);
  \draw[thick] (v2) to[bend left=50] (v9);
  \draw[thick] (v3) to[bend left=50] (v10);
  \draw[thick] (v4) to[bend left=50] (v11);
  \draw[thick] (v5) to[bend left=50] (v12);
  \draw[thick] (v6) to[bend left=50] (v13);
  \draw[thick] (v7) to[bend left=50] (v14);
\end{tikzpicture}
\\[2pt]
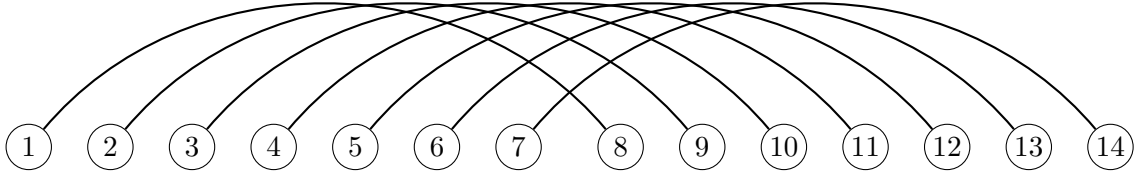
{\small $\Delta_{G_{14}}$ at $r=n-4=5$: seven disjoint edges, each
of gap exactly $m/2=7$}
\end{center}

Every prediction of Conjecture~\ref{conj:Cn-full-CM} matches this
table exactly: CM for $2\le r\le4$, a single failure at $r=5=n-4$,
and CM again from $r=6=n-3$ onward once the graph completes.
\end{example}

\subsection{Complement of cyclic graph under $\mathrm{Inc}(\mathbb{N})$}
 
We now turn to $C_n^c$, the complement of the cyclic graph $C_n$
studied in the previous subsection: two vertices are adjacent in
$C_n^c$ exactly when they are \emph{not} adjacent in $C_n$.
Throughout this subsection, for $n\ge4$ we consider $I_n$ to be the
edge ideal of the complement of the cyclic graph
\[
C_n^c := \big(\, \{i,j\}\ \mid\ 1\le i<j\le n,\ \ 2\le j-i\le n-2 \,\big\},\ [n] \big),
\]
i.e.\ $\{i,j\}$ ($i<j$) is an edge of $C_n^c$ precisely when
$2\le j-i\le n-2$. The restriction $n\ge4$ ensures $C_n^c$ has at
least one edge: for $n=3$ the range $[2,n-2]=[2,1]$ is empty, so
$C_3^c$ is edgeless. As in the previous subsections, we set
\[
I_{n+r} := \mathrm{Inc}(\mathbb{N})_{n,n+r}(I_n),
\]
let $G_{n+r}$ denote the corresponding graph of $I_{n+r}$, and we
first describe the generators of $I_{n+r}$, then compute the height
and unmixedness of the base case $I_n$, and finally characterize
Cohen--Macaulayness of the full chain.
 
\begin{lemma}\label{lem:Cnc-generators}
Let $n\ge4$ and $r\ge0$. Then
\[
\mathrm{Inc}(\mathbb{N})_{n,n+r}(I(C_n^c)) = I(C_{n+r}^c).
\]
\end{lemma}
 
\begin{proof}
Write $m=n+r$. We show the two inclusions separately.
 
\smallskip
\noindent($\subseteq$) Every generator of
$\mathrm{Inc}(\mathbb{N})_{n,n+r}(I(C_n^c))$ is $\pi(x_ix_j)=x_ax_b$
for some edge $\{i,j\}$ of $C_n^c$ (so $2\le j-i\le n-2$) and some
$\pi\in\mathrm{Inc}(\mathbb{N})_{n,m}$, with $a=\pi(i)$, $b=\pi(j)$.
Since $\pi$ is strictly increasing, $b-a\ge j-i\ge2$. For the upper
bound: the only pair in $[m]$ with gap $m-1$ is $(1,m)$, and
realizing it would force $\pi(i)=1$, $\pi(j)=m$; since $\pi$ maps
$[n]$ into $[m]$ monotonically, this forces $i=1$ and $j=n$, i.e.\
the domain edge would have to be $\{1,n\}$. But $\{1,n\}$ has gap
$n-1$, which is \emph{not} an edge of $C_n^c$ (it is precisely the
wrap edge of $C_n$ itself). Hence gap $m-1$ is never reached, so
$2\le b-a\le m-2$, i.e.\ $x_ax_b$ is a generator of $I(C_m^c)$.
 
\smallskip
\noindent($\supseteq$) Let $1\le a<b\le m$ with $d:=b-a\in[2,m-2]$;
we exhibit $\pi\in\mathrm{Inc}(\mathbb{N})_{n,m}$ and an edge
$\{i,j\}$ of $C_n^c$ with $\pi(i)=a$, $\pi(j)=b$.
 
\emph{Case $2\le d\le n-2$.} Choose the domain edge $\{i,i+d\}$ of
gap exactly $d$ (valid, as $d\le n-2$), with $i$ chosen so that the
remaining $n-2$ domain points fit into $[1,a-1]$ and $[b+1,m]$; this
is possible by the same room-counting argument used for the analogous
lemma on $L_n^c$.
 
\emph{Case $n-1\le d\le m-2$.} The domain edge of maximum gap $n-2$
is $\{i,i+n-2\}$ for $i\in\{1,2\}$. If $b\le m-1$, take $i=1$; if
$b=m$, take $i=2$ (valid since $d\le m-2$ forces $a=b-d\ge2$ in this
case). Set $j=i+n-2$, $\pi(i)=a$, $\pi(j)=b$, map the $n-3$ interior
domain points increasingly onto any $n-3$ of the $d-1\ge n-2$
integers strictly between $a$ and $b$, map the $i-1$ domain points
before position $i$ increasingly into $[1,a-1]$, and map the $n-j$
domain points after position $j$ increasingly into $[b+1,m]$ (all
possible by the choice of $i$). This $\pi$ is strictly increasing
and $\pi(x_ix_j)=x_ax_b$.
 
In both cases $x_ax_b\in\mathrm{Inc}(\mathbb{N})_{n,n+r}(I(C_n^c))$,
completing the proof.
\end{proof}
 
\begin{proposition}\label{prop:Cnc-height}
Let $n\ge4$. Then $\operatorname{ht}(I(C_n^c))=n-2$, and $I(C_n^c)$
is unmixed.
\end{proposition}
 
\begin{proof}
A pair $\{a,b\}$ ($a<b$) is independent in $C_n^c$ iff it is not an
edge, i.e.\ iff $b-a=1$ or $b-a=n-1$; the latter forces $a=1,b=n$.
So every independent pair is either a consecutive pair $\{k,k+1\}$
or the pair $\{1,n\}$ -- exactly the edges of $C_n$.
 
No independent set of size $\ge3$ exists: suppose $\{i,j,k\}$ is
independent with $i<j<k$ in $[n]$. Since $j-i\ge1$ and $k-j\ge1$,
we have $k-i\ge2$; for $\{i,k\}$ to be independent it must then have
gap $n-1$ exactly (the only alternative to a gap in $[2,n-2]$), which
forces $i=1$ and $k=n$. Now $\{1,j\}$ and $\{j,n\}$ must both be
independent with $1<j<n$: since $j-1<n-1$ and $n-j<n-1$, neither gap
can equal $n-1$, so both must equal $1$, forcing $j=2$ and $j=n-1$
simultaneously -- impossible for $n\ge4$. Hence $\alpha(C_n^c)=2$, so
by the height--independence identity,
$\operatorname{ht}(I(C_n^c))=n-\alpha(C_n^c)=n-2$.
 
For unmixedness, every vertex $k\in[n]$ has an independent partner:
if $2\le k\le n-1$, then $\{k-1,k\}$ is independent; if $k=1$, both
$\{1,2\}$ and $\{1,n\}$ are independent; if $k=n$, both $\{n-1,n\}$
and $\{1,n\}$ are independent. So no facet of $\Delta_{C_n^c}$ has
size $1$, and combined with $\alpha(C_n^c)=2$, every facet has size
exactly $2$: $\Delta_{C_n^c}$ is pure, so $I(C_n^c)$ is unmixed.
\end{proof}
 
\begin{theorem}\label{thm:Cnc-CM}
Let $n\ge4$ and $r\ge0$. Then
$\mathrm{Inc}(\mathbb{N})_{n,n+r}(I(C_n^c))$ is Cohen--Macaulay.
\end{theorem}
 
\begin{proof}
Set $m=n+r\ge4$. By Lemma~\ref{lem:Cnc-generators},
$\mathrm{Inc}(\mathbb{N})_{n,n+r}(I(C_n^c))=I(C_m^c)$, so it suffices
to show $I(C_m^c)$ is Cohen--Macaulay. By
Proposition~\ref{prop:Cnc-height} (applied with $m$ in place of $n$),
$\Delta_{C_m^c}$ is pure of dimension $1$, with facets exactly the
independent pairs described there: the consecutive pairs $\{k,k+1\}$
for $1\le k\le m-1$, together with the pair $\{1,m\}$. Viewed as a
graph, this facet set is precisely the edge set of $C_m$ itself,
which is connected, being a cycle. Hence $\Delta_{C_m^c}$ is pure and
connected, so shellable, so Cohen--Macaulay. Therefore $I(C_m^c)$ --
and hence $\mathrm{Inc}(\mathbb{N})_{n,n+r}(I(C_n^c))$ -- is
Cohen--Macaulay, for every $n\ge4$ and $r\ge0$.
\end{proof}
 
\begin{remark}
Unlike the cyclic-graph chain of the previous subsection, whose
Cohen--Macaulayness is governed by the delicate threshold of
Conjecture~\ref{conj:Cn-full-CM} (holding on a range of $r$ but
failing at the single value $r=n-4$), the complement-of-cyclic-graph
chain is Cohen--Macaulay \emph{unconditionally}, for every $r\ge0$.
This mirrors the contrast already seen between line graphs and their
complements: in both cases, passing to the complement removes the
threshold phenomenon entirely, because the independence complex of
the complement is literally (a copy of) the original graph itself,
which is connected by construction.
\end{remark}

 \section*{AI Use Declaration}
During the preparation of this manuscript, the authors used Claude
(Anthropic) to assist with formalizing proofs, checking algebraic and
combinatorial computations, drafting and refining LaTeX exposition,
and checking the internal consistency of definitions, lemmas, and
cross-references. The use of the tool was limited to assisting with
these tasks and did not include the independent generation of
research ideas, conjectures, or mathematical results. All mathematical
content, including every theorem and proof, was independently reviewed
and verified by the authors, who take full responsibility for the
correctness and integrity of the manuscript. Claude is not an author
or co-author of this work.\\

{\bf Acknowledgments.} {The first and second authors expressed their heartfelt appreciation to the Graduate Research Assistantship in Developing Countries (GRAID) Program of the IMU/CDC for their invaluable support and assistance. }


\bibliographystyle{amsalpha}
\bibliography{References}

@article{aschenbrenner2007finite,
  title={Finite generation of symmetric ideals, Trans. Am. Math. Soc., 359: 11},
  author={Aschenbrenner, M and Hillar, Ch J},
  year={2007}
}

@article{draisma2010finiteness,
  title={Finiteness for the k-factor model and chirality varieties},
  author={Draisma, Jan},
  journal={Advances in Mathematics},
  volume={223},
  number={1},
  pages={243--256},
  year={2010},
  publisher={Elsevier}
}

@article{cohen1967laws,
  title={On the laws of a metabelian variety},
  author={Cohen, Daniel E},
  journal={Journal of Algebra},
  volume={5},
  number={3},
  pages={267--273},
  year={1967},
  publisher={Academic Press}
}

@article{church2015fi,
  title={FI-modules and stability for representations of symmetric groups},
  author={Church, Thomas and Ellenberg, Jordan S and Farb, Benson},
  year={2015}
}

@article{nagel2019fi,
  title={FI-and OI-modules with varying coefficients},
  author={Nagel, Uwe and R{\"o}mer, Tim},
  journal={Journal of Algebra},
  volume={535},
  pages={286--322},
  year={2019},
  publisher={Elsevier}
}

@article{putman2017representation,
  title={Representation stability and finite linear groups},
  author={Putman, Andrew and Sam, Steven V},
  year={2017}
}

@article{brouwer2011equivariant,
  title={Equivariant Gr{\"o}bner bases and the Gaussian two-factor model},
  author={Brouwer, Andries and Draisma, Jan},
  journal={Mathematics of Computation},
  volume={80},
  number={274},
  pages={1123--1133},
  year={2011}
}

@article{conca2014noetherianity,
  title={Noetherianity up to symmetry},
  author={Conca, Aldo and Di Rocco, Sandra and Draisma, Jan and Huh, June and Sturmfels, Bernd and Viviani, Filippo and Draisma, Jan},
  journal={Combinatorial Algebraic Geometry: Levico Terme, Italy 2013, Editors: Sandra Di Rocco, Bernd Sturmfels},
  pages={33--61},
  year={2014},
  publisher={Springer}
}

@article{draismanoetherianity,
  title={Noetherianity for infinite-dimensional toric varieties. 2013},
  author={Draisma, Jan and Eggermont, Rob H and Krone, Robert and Leykin, Anton},
  journal={Preprint available from http://arxiv. org/abs/1306.0828}
}

@article{hillar2012finite,
  title={Finite Gr{\"o}bner bases in infinite dimensional polynomial rings and applications},
  author={Hillar, Christopher J and Sullivant, Seth},
  journal={Advances in Mathematics},
  volume={229},
  number={1},
  pages={1--25},
  year={2012},
  publisher={Elsevier}
}

@article{hocsten2007finiteness,
  title={A finiteness theorem for Markov bases of hierarchical models},
  author={Ho{\c{s}}ten, Serkan and Sullivant, Seth},
  journal={Journal of Combinatorial Theory, Series A},
  volume={114},
  number={2},
  pages={311--321},
  year={2007},
  publisher={Elsevier}
}

@article{nagel2017equivariant,
  title={Equivariant Hilbert series in non-noetherian polynomial rings},
  author={Nagel, Uwe and R{\"o}mer, Tim},
  journal={Journal of Algebra},
  volume={486},
  pages={204--245},
  year={2017},
  publisher={Elsevier}
}

@article{le2021castelnuovo,
  title={Castelnuovo--Mumford regularity up to symmetry},
  author={Le, Dinh Van and Nagel, Uwe and Nguyen, Hop D and R{\"o}mer, Tim},
  journal={International Mathematics Research Notices},
  volume={2021},
  number={14},
  pages={11010--11049},
  year={2021},
  publisher={Oxford University Press}
}

@article{le2020codimension,
  title={Codimension and projective dimension up to symmetry},
  author={Le, Dinh Van and Nagel, Uwe and Nguyen, Hop D and R{\"o}mer, Tim},
  journal={Mathematische Nachrichten},
  volume={293},
  number={2},
  pages={346--362},
  year={2020},
  publisher={Wiley Online Library}
}

@article{almousa2022alexander,
  title={Alexander duals of symmetric simplicial complexes and Stanley-Reisner Ideals},
  author={Almousa, Ayah and Bruegge, Kaitlin and Juhnke-Kubitzke, Martina and Nagel, Uwe and Pevzner, Alexandra},
  journal={arXiv preprint arXiv:2209.14132},
  year={2022}
}

@article{murai2020betti,
  title={Betti tables of monomial ideals fixed by permutations of the variables},
  author={Murai, Satoshi},
  journal={Transactions of the American Mathematical Society},
  volume={373},
  number={10},
  pages={7087--7107},
  year={2020}
}

@article{hoa2024asymptotic,
  title={Asymptotic depth of invariant chains of edge ideals},
  author={Hoa, Tran Quang and Hoang, Do Trong and Van Le, Dinh and Nguyen, Hop D and Nguyen, Thai Thanh},
  journal={arXiv preprint arXiv:2409.06252},
  year={2024}
}

@article{van2024regularity,
  title={On Regularity and Projective Dimension of Invariant Chains of Monomial Ideals},
  author={Van Le, Dinh and Nguyen, Hop D},
  journal={Michigan Mathematical Journal},
  volume={1},
  number={1},
  pages={1--22},
  year={2024},
  publisher={University of Michigan, Department of Mathematics}
}

@article{hoang2024asymptotic,
  title={Asymptotic regularity of invariant chains of edge ideals},
  author={Hoang, Do Trong and Nguyen, Hop D and Tran, Quang Hoa},
  journal={Journal of Algebraic Combinatorics},
  volume={59},
  number={1},
  pages={55--94},
  year={2024},
  publisher={Springer}
}

@article{kahle2022invariant,
  title={Invariant chains in algebra and discrete geometry},
  author={Kahle, Thomas and Le, Dinh Van and Römer, Tim},
  journal={SIAM Journal on Discrete Mathematics},
  volume={36},
  number={2},
  pages={975--999},
  year={2022},
  publisher={SIAM}
}

@article{hoang2026invariant,
  title={Invariant chains of graphs},
  author={Hoang, Do Trong and Koley, Mitra and Van Le, Dinh},
  journal={arXiv preprint arXiv:2608.17354},
  year={2026}
}

@article{juhnke2018asymptotic,
  title={Asymptotic behavior of symmetric ideals: a brief survey},
  author={Juhnke-Kubitzke, Martina and Le, Dinh Van and R{\"o}mer, Tim},
  journal={National School on Algebra},
  pages={73--94},
  year={2018},
  publisher={Springer}
}
\end{document}